\documentclass[12pt]{article}

\usepackage{fullpage}

\RequirePackage{amsthm,amsmath,amsfonts,amssymb}

\theoremstyle{plain} \newtheorem{theorem}{Theorem}[section]
\theoremstyle{plain} 
\theoremstyle{plain} \newtheorem{lemma}[theorem]{Lemma}
\theoremstyle{plain} \newtheorem{corollary}[theorem]{Corollary}
\theoremstyle{definition} \newtheorem{definition}[theorem]{Definition}
\theoremstyle{definition} 
\theoremstyle{remark} \newtheorem{remark}[theorem]{Remark}
\theoremstyle{remark} \newtheorem{example}[theorem]{Example}

\newcommand{\R}{\mathbb{R}}

\newcommand{\Sphere}{\mathbb{S}^{d-1}}
\newcommand{\neutral}{\mathbf{e}}

\newcommand{\sK}{{\mathcal{K}}}

\newcommand{\sC}{{\mathcal{C}}}

\newcommand{\KK}{{\mathbb{K}}}

\newcommand{\sCo}{\sC_0}
\newcommand{\sKo}{\sK_0}

\newcommand{\eps}{\varepsilon}
\DeclareMathOperator{\Cvx}{Cvx}

\newcommand{\thf}{\frac{1}{2}}

\usepackage{fancybox}
\newlength{\querylen}
\begin{document}

\title{Continued fractions built from convex sets and convex
  functions}

\author{Ilya Molchanov\footnote{Supported by the Swiss National
    Science Foundation grant 200021-137527 and the Santander Bank
    through
    the Chair of Excellence programme at the University Carlos III of Madrid.}\\
  \normalsize Institute of Mathematical Statistics and Actuarial
  Science\\
  \normalsize University of Bern, Sidlerstrasse 5, 3012 Bern, Switzerland\\
  \normalsize ilya.molchanov@stat.unibe.ch}

\date{\today}
\date{}
\maketitle

\begin{abstract}
  In a partially ordered semigroup with the duality (or polarity)
  transform, it is possible to define a generalisation of continued
  fractions. General sufficient conditions for convergence of
  continued fractions are provided.  Two particular applications
  concern the cases of convex sets with the Minkowski addition and the
  polarity transform
  and the family of
  non-negative convex functions with the Legendre--Fenchel and
  Artstein-Avidan--Milman transforms.


\end{abstract}

\section{Introduction}
\label{sec:introduction}

The studies of order reversing involutions (also called dualities or
polarities) on partially ordered spaces have recently gained a
considerable attention. If $\KK$ is a partially ordered space, then the
map from $x\in \KK$ to $x^*\in\KK$ is said to be an order reversing
involution if $x^{**}=x$ for all $x$ and $x\leq y$ implies that
$y^*\leq x^*$.

The two main examples are the family of convex sets containing the
origin and ordered by inclusion and the family of convex functions on
$\R^d$ ordered pointwisely. It is shown in
\cite{art:mil08t,boer:sch08} that the only (up to a rigid motion)
order reversing involution on the family of compact convex sets
containing the origin is the classical polar transform, see
\cite{schn2} and Section~\ref{sec:cont-fract-conv}. For the family of
non-negative convex functions on $\R^d$ that vanish at the origin, only
two involutions (up to rigid motions) exist: one is the classical
Legendre--Fenchel transform \cite{roc70} and the other is the 
A-transform, see \cite{art:mil09,art:mil11h} and
Section~\ref{sec:space-non-negative}.

It is also possible to endow the space $\KK$ with an addition
operation that turns it into an abelian semigroup. Such an addition
may be chosen to be the lattice operation corresponding to the order
or defined otherwise. For example, on the family of closed convex sets
partially ordered by inclusion it is possible to consider the convex
hull of the union as the semigroup operation or add sets using the
Minkowski (elementwise) addition. In the case of convex functions, a
natural semigroup operation is the arithmetic addition, while it is
also possible to consider the epigraphical or level sums. In this
paper, it is assumed that $\KK$ is an abelian semigroup with the
additive operation that is consistent with the order.

The order reversing nature of the involution makes possible to
define a continued fraction on a semigroup. In comparison with
classical numerical continued fractions, the additive operation is the
semigroup addition, while the one over operation is replaced by the
involution. The classical concept of continued fraction is recovered
for the semigroup $[0,\infty]$ with the conventional addition and the
involution given by the arithmetic inverse. The semigroup setting 
differs from the setting of continued fractions in Jordan algebras
pursued in \cite{bern95} and the studies of multidimensional continued
fractions in \cite{karp13} and \cite{kon:suh99}. 

This paper argues that convergence results for continued fractions
built from convex sets and those from convex functions can be derived
from general statements concerning continued fractions on
semigroups. The non-existence of an inverse operation to the addition
renders impossible the direct use of most of the tools from the
classical theory of continued fractions, see
\cite{jon:thr80}. Instead, the key emphasis is put on the partial
order together with bounds on the Lipschitz constant for the
involution transform. These bounds are needed for properly chosen
subsets of the original semigroup, e.g. for convex sets that contain
the unit ball. The latter in the abstract setting becomes an
involution-invariant element, which is also the key ingredient to
define a suitable metric. In this very general setting, several
sufficient conditions for convergence of continued fractions with
constant and variable terms are obtained in
Section~\ref{sec:convergence-results}. A particular attention is
devoted to periodic continued fractions, whose limits may be regarded
as a generalisation of quadratic irrational numbers in the semigroup
setting.

The general results are applied for set-valued continued fractions in
Section~\ref{sec:cont-fract-conv}. For example, a continued fraction
with a constant term being a set $K$ converges if $K$ is sandwiched
between two Euclidean balls with diameters $r$ and $R$ such that
either $r>1$, or $r=1$ and $R$ is finite, or $r<1$ and $R<r/(1-r)$,
see Theorem~\ref{thr:constant}. In the set-valued case, we obtain a
necessary and sufficient condition for the convergence that amounts to
the fact that an odd-numbered approximant of the continued fraction is
a subset of the ball with radius strictly less than one. The key
argument is the Lipschitz property of the polarity transform meaning
that the Hausdorff distance between the polars of two convex sets
containing a centred ball of radius $r$ is bounded by $r^{-2}$ times
the Hausdorff distance between the original sets.

Section~\ref{sec:space-non-negative} presents several convergence
results for continued fractions of convex functions both for the
Legendre--Fenchel and A-transforms. It is rather easy to
modify these results to apply to the semigroup of log-concave
functions with the multiplication. Another possible application left
outside of the framework of this paper is for the semigroup of
probability measures with the convolution operation as the addition
and the involution inherited by an application of the involution
operation on the original space.

\section{Continued fractions on semigroups}
\label{sec:cont-fract-semigr}

Let $\KK$ be a \emph{partially ordered abelian semigroup} with the
neutral element $\neutral$. Assume that $\neutral\leq x$ for all
$x\in\KK$ and that the order is compatible with the addition, that is
$x\leq y$ for $x,y\in\KK$ implies $x+z\leq y+z$ for all
$z\in\KK$. Assume that, the order is weaker than the semigroup order,
i.e.  $y=x+z$ for some $z\in\KK$ yields that $x\leq y$.  In this case,
the semigroup is reduced, meaning that the only invertible element is
$\neutral$.

Assume that $\KK$ is equipped with an \emph{order reversing
  involution} $x\mapsto x^*$ (also called duality or polarity
transform), so that $x^{**}=x$ and $x\leq y$ implies that $x^*\geq
y^*$. The involution is not assumed to commute with the addition,
i.e. $(x+y)^*$ is not necessarily $x^*+y^*$.  Since $\neutral\leq x$
for each $x\in\KK$, we have $x^*\leq \neutral^*$ meaning that
$\neutral^*$ dominates all elements from $\KK$ and
$\neutral^*+x=\neutral^*$ for all $x\in\KK$.

Consider a sequence $\{x_n,n\geq 1\}$ of elements from $\KK$ and
define the sequence $[x_1,\dots,x_n]$, $n\geq1$, recursively by
letting
\begin{align*}
  [x_1]&=x_1^*,\\
  [x_1,\dots,x_{n+1}]&=(x_1+[x_2,\dots,x_{n+1}])^*,\quad n\geq 1\,.
\end{align*}
The element $z_n=[x_1,\dots,x_n]$ is said to be the \emph{$n$th
  approximant} of the \emph{continued fraction} generated by
$\{x_n,n\geq1\}$.

\begin{remark}
  \label{rem:dual-addition}
  The dual operation to the addition is defined by setting
  \begin{displaymath}
    x\oplus y=(x^*+y^*)^*\,.
  \end{displaymath} 
  Then $[x_1,\dots,x_n]=x_1^*\oplus[x_2,\dots,x_n]^*$.
\end{remark}

\begin{example}
  If $\KK=[0,\infty]$ with the conventional addition and involution
  $x^*=\frac{1}{x}$, then
  \begin{displaymath}
    z_n=[x_1,\dots,x_n]=\frac{1}{x_1+\frac{1}{x_2+\frac{1}{x_3+\cdots+\frac{1}{x_n}}}}
  \end{displaymath}
  is the $n$th approximant of the numerical continued fraction, see
  e.g. \cite{jon:thr80}.
\end{example}

\begin{remark}
  The neutral element and its dual influence the continued fraction as
  follows
  \begin{align*}
    [x_1,\dots,x_m,\neutral^*,x_{m+1},\dots,x_n]
    &=[x_1,\dots,x_m],\\
    [x_1,\dots,x_m,\neutral,x_{m+1},\dots,x_n]
    &=[x_1,\dots,x_m+x_{m+1},\dots,x_n]\,,\quad m\neq 0,n,\\
    [\neutral,x_1,\dots,x_n] &=[x_1,\dots,x_n]^*=x_1+[x_2,\dots,x_n],\\
    [x_1,\dots,x_n,\neutral] &=[x_1,\dots,x_{n-1}].
  \end{align*}
\end{remark}

The following result generalises the well-known property of continued
fractions with positive terms.

\begin{lemma}
  \label{lemman:zn}
  The $n$th approximant $z_n=[x_1,\dots,x_n]$ is increasing in each of
  the even numbered terms and decreasing in each of the odd numbered
  terms.  The sequence $\{z_{2m},m\geq1\}$ is increasing and
  $\{z_{2m-1},m\geq1\}$ is decreasing.
\end{lemma}
\begin{proof}
  A direct check shows that $[x_1,\dots,x_k,\dots,x_n]$ is increasing
  in $x_k$ if $k$ is even and decreasing if $k$ is odd.  Then
  \begin{displaymath}
    z_{2m+2}=[x_1,\dots,x_{2m},(x_{2m+1}+x_{2m+2}^*)^*]
    \geq [x_1,\dots,x_{2m},\neutral^*]
    =[x_1,\dots,x_{2m}]=z_{2m}\,.
  \end{displaymath}
  A similar argument applies to the odd part of the continued
  fraction.
\end{proof}

From now on, assume that $\KK$ is equipped with the \emph{scaling
  transformation} $x\mapsto ax$ by positive real numbers $a$. It is
assumed that the scaling satisfies the distributivity laws and that
$(ax)^*=a^{-1}x^*$ for all $a>0$ and $x\in\KK$. In particular, the
second distributivity law implies that $ax\leq bx$ for $a\leq b$.

Fix any $h\in\KK$ such that $h^*=h$ (in this case $h$ is said to be
\emph{self-polar}) and define, for $x,y\in\KK$,
\begin{equation}
  \label{eq:rhoh}
  \rho_h(x,y)=\inf\{t\geq 0:\; x\leq y+th,\; y\leq x+th\}\,.
\end{equation}
Note that $\rho_h$ is a semimetric that might take infinite values and
is scale-homogeneous, that is $\rho_h(ax,ay)=a\rho_h(x,y)$ for
$a>0$. The (possibly infinite) norm of $x\in\KK$ is defined as
$\|x\|_h=\rho_h(x,\neutral)$. 

Let $\KK_h$ be the family of $x\in\KK$ such that $x\leq ah$ for some
$a>0$. Note that $\KK_h$ is a sub-semigroup of $\KK$.  Since
$\rho_h(x,y)\leq a$ for $x\leq ah$ and $y\leq ah$, $\rho_h(x,y)$ takes
finite values for $x,y\in\KK_h$. If $x\in\KK_h$, then
$\rho_h(a_nx,\neutral)\to0$ as $a_n\downarrow0$. Denote
\begin{displaymath}
  \KK_h^*=\{x\in\KK:\;x^*\in\KK_h\}\,.
\end{displaymath}
It is sensible to let $0h=\neutral$ and $\infty
h=\neutral^*$.

In the following we assume that 
\begin{equation}
  \label{eq:3}
  \bigcap_{t>0} \{x\in\KK:\; x\leq y+th\}=\{x\in\KK:\; x\leq y\}.
\end{equation}

\begin{lemma}
  \label{lemma:order}
  If \eqref{eq:3} holds, then $\rho_h$ is a metric on $\KK_h$ and the
  order is closed, that is the set $\{(x,y):\; x\leq y\}$ is closed in
  the product space $(\KK_h,\rho_h)\times(\KK_h,\rho_h)$.
\end{lemma}
\begin{proof}
  Because of \eqref{eq:3}, $\rho_h(x,y)=0$ yields that $x\leq y$ and
  $y\leq x$, so that $x=y$. Other properties of the metric are
  evidently satisfied.  If $\rho_h(x_n,x)\to 0$ and $\rho_h(y_y,y)\to
  0$ and $x_n\leq y_n$ for all $n$, then for a sequence $t_n\to0$ one
  has
  \begin{displaymath}
    x\leq x_n+t_nh\leq y_n+t_nh\leq y+2t_nh\,,
  \end{displaymath}
  so that $x\leq y$ by \eqref{eq:3}.
\end{proof}

\begin{remark}[Multiple self-polar elements]
  If $h_1$ and $h_2$ are two distinct self-polar elements, then
  $h_1\leq a h_2$ either holds for some $a>1$ or does not hold for any
  $a>0$. Indeed, by passing to the polars, the inequality becomes
  $h_1\geq a^{-1}h_2$, so that $a^{-1}h_2\leq h_1\leq ah_2$. If
  $h_1\leq a h_2$, then $a^{-1}\rho_{h_1}(x,y)\leq \rho_{h_2}(x,y)\leq
  a\rho_{h_1}(x,y)$ for all $x,y\in\KK$.
\end{remark}

\begin{definition}
  The continued fraction generated by a sequence $x_n\in\KK$,
  $n\geq1$, is said to converge if $z_n=[x_1,\dots,x_n]$ converges in
  $\rho_h$ as $n\to\infty$ to an element of $\KK$.
\end{definition}

\begin{example}
  If $x_n=b_n h$ for a self-polar $h\in\KK$ and non-negative real
  numbers $b_n$, $n\geq1$, then $z_n=[b_1,\dots,b_n]h$, so that the
  convergence of $z_n$ can be derived from the convergence of the
  numerical continued fractions. For instance, the Seidel-Stern
  theorem asserts that $z_n$ converges if and only if $\sum
  b_n=\infty$. While it is tempting to conjecture that the continued
  fraction in $\KK$ converges if $\sum x_n=\neutral^*$,
  Example~\ref{ex:seidel} shows that this is wrong.
\end{example}

\begin{lemma}
  \label{lem:bound}
  Let $r_nh\leq x_n\leq R_nh$, $n\geq1$, for $h\in\KK$ such that
  $h^*=h$. Then
  \begin{displaymath}
    [R_1,r_2,\dots,a_n]h\leq [x_1,\dots,x_n]
    \leq [r_1,R_2,\dots,b_n]h\,, 
  \end{displaymath}
  where $a_n=r_n$ and $b_n=R_n$ if $n$ is even and $a_n=R_n$ and
  $b_n=r_n$ if $n$ is odd. 
\end{lemma}
\begin{proof}
  It suffices to use the induction argument based on 
  \begin{align*}
    [R_1h,x_2,\dots,x_n]=
    (R_1h+[x_2,\dots,x_n])^*
    &\leq [x_1,\dots,x_n]\\
    &\leq (r_1h+[x_2,\dots,x_n])^*
    =[r_1h,x_2,\dots,x_n]
  \end{align*}
  for $n\geq1$.
\end{proof}

\section{Convergence results}
\label{sec:convergence-results}

The key technical condition used to deduce the convergence of
continued fractions in $\KK$ requires that
\begin{equation}
  \label{eq:xh}
  \rho_h(x^*,(x+th)^*)\leq t,\qquad t>0,
\end{equation}
if $h\leq x$.
A weaker variant of this condition is 
\begin{equation}
  \label{eq:xhc}
  \rho_h(x^*,(x+th)^*)\leq C_R^2t,\qquad t>0,
\end{equation}
if $h\leq x\leq Rh$, where $C_R$ is a positive finite function of
$R\in[1,\infty]$. Substituting $x=h$ shows that $C_R\geq C_1=1$. It is
immediate that the function $C_R$ can be chosen to be non-decreasing
and right-continuous. Let $C_\infty$ be the constant in the right-hand
side of \eqref{eq:xhc} that ensures the inequality for all $x\geq h$.
Note that $\rho_h(x^*,(x+th)^*)=\rho_h([x],[x,t^{-1}h])$.

\begin{lemma}
  \label{lemma:lip}
  If \eqref{eq:xhc} holds, then 
  \begin{displaymath}
    \rho_h(x^*,y^*)\leq C_{R/r}^2r^{-2}\rho_h(x,y)
  \end{displaymath}
  for all $x,y\in\KK$ such that $rh\leq x\leq Rh$ and $rh\leq y\leq
  Rh$.
\end{lemma}
\begin{proof}
  Let $\rho_h(x,y)\leq t$. Then $x\leq y+t h$
  and $y\leq x+th$, so that
  \begin{displaymath}
    \rho_h(x^*,y^*)\leq 
    \max(\rho_h(x^*,(x+t h)^*),\rho_h(y^*,(y+t h)^*))\,.
  \end{displaymath}
  It suffices to note that \eqref{eq:xhc} and the scaling property of
  $\rho_h$ imply that
  \begin{displaymath}
    \rho_h(x^*,(x+th)^*)\leq C_{R/r}^2r^{-2}t,\qquad t>0,
  \end{displaymath}
  for all $x$ such that $rh\leq x\leq Rh$.  
\end{proof}

\begin{corollary}
  \label{cor:inv-continuous}
  If \eqref{eq:xhc} holds, then the involution operation is
  $\rho_h$-continuous on $\KK_h\cap\KK_h^*$ and, for $x,y\in\KK_h$,
  \begin{displaymath}
    \rho_h(x^*,y^*)\leq C_\infty^2 \max(\|x^*\|_h,\|y^*\|_h)^2\rho_h(x,y)\,.
  \end{displaymath}
\end{corollary}

The following result covers the general case of continued fractions
with variable terms.

\begin{theorem}
  \label{thr:subk}
  Assume that $(\KK_h,\rho_h)$ is complete and condition
  \eqref{eq:xhc} holds. Assume that there exist $k\geq1$ and $0<a\leq
  b\leq \infty$ such that $C_{b/a}<a$ and
  \begin{align}
    \label{eq:xmk}
    x_n+[x_{n+1},\dots,x_{n+2k}]\geq ah\,,\\
    \label{eq:xmk2}
    x_n+[x_{n+1},\dots,x_{n+2k-1}]\leq bh\,,
  \end{align}
  for all $n\geq1$.  If, for some $l\geq 0$,
  \begin{equation}
    \label{eq:limsupm}
    \limsup_{n\to\infty} 
    \rho_h([x_{n-2k-l+1},\dots,x_n],[x_{n-2k-l+1},\dots,x_{n+1}])<\infty\,,
  \end{equation}
  then the continued fraction $z_n=[x_1,\dots,x_n]$ converges to
  $z\in\KK_h$.
\end{theorem}
\begin{proof}
  By Lemma~\ref{lemman:zn}, for all $m\geq1$ and $n\geq 2k+m$,
  \begin{align*}
    x_m+[x_{m+1},\dots,x_{m+2k},x_{m+2k+1},\dots,x_n]
    &\geq x_m+[x_{m+1},\dots,x_{m+2k},\neutral^*,\dots,x_n]\\
    &=x_m+[x_{m+1},\dots,x_{m+2k}]\geq ah\,,
  \end{align*}
  and 
  \begin{align*}
    x_m+[x_{m+1},\dots,x_{m+2k},x_{m+2k+1},\dots,x_n]
    &\leq x_m+[x_{m+1},\dots,x_{m+2k-1},\neutral^*,\dots,x_n]\\
    &=x_m+[x_{m+1},\dots,x_{m+2k-1}]\leq bh\,.
  \end{align*}
  In particular, $z_n\leq a^{-1}h$ and so $z_n\in\KK_h$ for
  sufficiently large $n$.  By Lemma~\ref{lemma:lip}, for $n\geq 2k+m$,
  \begin{align*}
    \rho_h([x_m,\dots,x_n],[x_m,\dots,x_{n+1}])
    &=\rho_h((x_m+[x_{m+1},\dots,x_n])^*,(x_m+[x_{m+1},\dots,x_{n+1}])^*)\\
    &\leq C_{b/a}^2 a^{-2} \rho_h([x_{m+1},\dots,x_n],[x_{m+1},\dots,x_{n+1}])\,.
  \end{align*}
  By iterating this argument for $m=1,\dots,n-2k-l$, we arrive at
  \begin{align*}
    \rho_h(z_n,z_{n+1})
    &\leq (a^{-1}C_{b/a})^{2(n-2k-l)}
    \rho_h([x_{n-2k-l+1},\dots,x_n],[x_{n-2k-l+1},\dots,x_{n+1}])\\
    &\leq (a^{-1}C_{b/a})^{2(n-2k-l)} c
  \end{align*}
  for all sufficiently large $n$ and some finite $c$ that dominates
  the upper limit in \eqref{eq:limsupm}.  Since the series
  $\sum\rho_h(z_n,z_{n+1})$ converges, the sequence $\{z_n\}$ is
  fundamental and its convergence follows from the completeness
  assumption.
\end{proof}

\begin{remark}
  \label{rem:xmk}
  In Theorem~\ref{thr:subk} it suffices to impose only \eqref{eq:xmk}
  with $a>C_\infty$, which becomes $a>1$ if \eqref{eq:xh} holds.
\end{remark}

\begin{corollary}
  \label{cor:epsilon-in}
  Assume that $(\KK_h,\rho_h)$ is complete and condition
  \eqref{eq:xhc} holds. Furthermore, assume that there exist $k\geq1$
  and $0<a\leq b\leq \infty$ such that $q=a^{-1}C_{b/a}<1$,
  \eqref{eq:xmk} and \eqref{eq:xmk2} hold, and there exists $r>0$ such
  that $x_n\geq rh$ for all $n\geq1$. Then the continued fraction
  $z_n=[x_1,\dots,x_n]$ converges to $z\in\KK_h$ and
  \begin{equation}
    \label{eq:2}
    \rho_h(z_n,z)\leq \frac{q^{2(n-2k)}}{1-q^2}r^{-1}\,.
  \end{equation}  
\end{corollary}
\begin{proof}
  It suffices to show that the upper limit in \eqref{eq:limsupm} for
  $l=0$ is bounded by $r^{-1}$.  Note that $[x_1,\dots,x_{n+1}]$
  lies between $[x_1,\dots,x_n,\neutral]$ and
  $[x_1,\dots,x_n,\neutral^*]$, whence
  \begin{align*}
    \rho_h([x_1,\dots,x_n]&,[x_1,\dots,x_{n+1}])\\
    &\leq
    \max(\rho_h([x_1,\dots,x_n],[x_1,\dots,x_n,\neutral]),
    \rho_h([x_1,\dots,x_n],[x_1,\dots,x_n,\neutral^*])\\
    &=\max(\rho_h([x_1,\dots,x_n],[x_1,\dots,x_{n-1}]),
    \rho_h([x_1,\dots,x_n],[x_1,\dots,x_n])\\
    &=\rho_h([x_1,\dots,x_n],[x_1,\dots,x_{n-1}])\,.
  \end{align*}
  By iterating this argument, we obtain
  \begin{displaymath}
    \rho_h([x_1,\dots,x_n],[x_1,\dots,x_{n+1}])\leq
    \rho_h(\neutral,[x_1])
    \leq r^{-1}\,.
  \end{displaymath}
  Therefore, $\rho(z_n,z_{n+1})\leq q^{2(n-2k)}r^{-1}$, whence
  \begin{displaymath}
    \rho_h(z_n,z_m)\leq \frac{q^{2(n-2k)}}{1-q^2}r^{-1}
  \end{displaymath}
  for all $m\geq n$ and it suffices to let $m\to\infty$. 
\end{proof}

The subsequent result relies on bounding the terms of the continued
fraction by $rh$ from below and $Rh$ from above.  In the following
denote
\begin{equation}
  \label{eq:upsilon}
  \upsilon(r,R)=[r,R,r,R,\dots]^{-1}=\thf (\sqrt{r^2+4r/R}+r)\,,
  \qquad r,R\in (0,\infty]\,,
\end{equation}
where the value of the numerical periodic continued fraction is found
by solving the recursive equation. 

\begin{theorem}
  \label{thr:uniform}
  Assume that $(\KK_h,\rho_h)$ is complete and condition
  \eqref{eq:xhc} holds. If $rh\leq x_n\leq Rh$ for all $n\geq1$ and
  $0<r\leq R\leq\infty$ such that
  \begin{equation}
    \label{eq:urr}
    \frac{C_{\upsilon(R,r)/\upsilon(r,R)}}{\upsilon(r,R)}<q<1\,,
  \end{equation}
  then the continued fraction $z_n=[x_1,\dots,x_n]$ converges to
  $z\in\KK_h$ and \eqref{eq:2} holds.
\end{theorem}
\begin{proof}
  By Lemma~\ref{lem:bound}, 
  \begin{displaymath}
    x_n+[x_{n+1},\dots,x_{n+2k}]\geq (r+[R,r,\dots,R,r])h\to \upsilon(r,R)h
    \quad\text{as }\; k\to\infty,
  \end{displaymath}
  and 
  \begin{displaymath}
    x_n+[x_{n+1},\dots,x_{n+2k-1}]\leq (R+[r,R,\dots,r])h\to \upsilon(R,r)h
    \quad\text{as }\; k\to\infty\,.
  \end{displaymath}
  Fix any $\eps>0$.  Then \eqref{eq:xmk} and \eqref{eq:xmk2} hold for
  sufficiently large $k$ with $a=\upsilon(r,R)-\eps$ and
  $b=\upsilon(R,r)+\eps$. Finally $C_{b/a}<a$ follows from
  \eqref{eq:urr} taking into account that the function $C_R$ is
  right-continuous. The convergence and the bound follow from
  Corollary~\ref{cor:epsilon-in}.
\end{proof}

\begin{corollary}
  \label{cor:uniform-simple}
  Assume that \eqref{eq:xh} holds and $(\KK_h,\rho_h)$ is a complete
  metric space.  Assume that $rh\leq x_n\leq Rh$ for all $n\geq1$ with
  $R\in[r,\infty]$, and either (i) $r>1$
  or (ii) $r=1$ and $R<\infty$ or (iii) $r<1$ and $R\leq
  r/(1-r)$. Then $z_n=[x_1,\dots,x_n]$ converges in the metric
  $\rho_h$ to $z\in \KK_h$ and \eqref{eq:2} holds with any $q>
  \upsilon(r,R)$.
\end{corollary}

Below we present another convergence condition that handles the case
when $x_n\geq r_nh$ with $\inf r_n=1$.

\begin{theorem}
  \label{thr:non-constant-R-r}
  Assume that \eqref{eq:xhc} holds and $(\KK_h,\rho_h)$ is complete.
  Let $\{x_n,n\geq1\}$ be a sequence of elements from $\KK$ such that
  $r_nh\leq x_n\leq R_nh$, $n\geq1$, where $R_n\in[r_n,\infty]$. If
  \begin{equation}
    \label{eq:limr}
    \liminf_{n\to \infty} n\log(r_nr_{n+1}C_{(R_n+r_{n+1}^{-1})/r_n}^{-2})>1,
  \end{equation}
  then the continued fraction $z_n=[x_1,\dots,x_n]$ converges.
\end{theorem}
\begin{proof}
  Note that 
  \begin{displaymath}
    r_ih \leq x_i+[x_{i+1},\dots,x_n]\leq (R_i+r_{i+1}^{-1})h
  \end{displaymath}
  for all $i\leq n-1$. Denote $a_i=C_{(R_i+r_{i+1}^{-1})/r_i}$,
  $i\geq1$. By Lemma~\ref{lemma:lip},
  \begin{align*}
    \rho_h(z_n,z_{n+1})
    &=\rho_h((x_1+[x_2,\dots,x_n])^*,(x_1+[x_2,\dots,x_{n+1}])^*)\\
    &\leq r_1^{-2}a_1^2\rho_h([x_2,\dots,x_n],[x_2,\dots,x_{n+1}])\\
    &\leq
    r_1^{-2}r_2^{-2}a_1^2a_2^2\rho_h([x_3,\dots,x_n],[x_3,\dots,x_{n+1}])\\
    &\leq \rho_h(x_n^*,(x_n+x_{n+1}^*)^*)\prod_{i=1}^{n-1} (r_i^{-2}a_i^2)
    \leq r_{n+1}^{-1}\prod_{i=1}^{n} (r_i^{-2}a_i^2) \,.
  \end{align*}
  In view of \eqref{eq:limr}, the logarithmic convergence criterion
  yields that the series $\sum \rho_h(z_n,z_{n+1})$ converges, so that
  the sequence $\{z_n\}$ is fundamental.
\end{proof}

\begin{remark}
  If \eqref{eq:xh} holds, then \eqref{eq:limr} becomes
  \begin{displaymath}
    \liminf_{n\to \infty} n\log(r_nr_{n+1})>1\,.
  \end{displaymath}
\end{remark}

Theorem~\ref{thr:uniform} and Corollary~\ref{cor:uniform-simple} cover
the important case of continued fractions
\begin{equation}
  \label{eq:zn-xxx}
  z_n=[\underbrace{x,\dots,x}_n]
\end{equation}
with constant terms.  An alternative proof of the convergence of
continued fractions with constant terms under the same conditions can
be carried over using the contraction mapping theorem.

For continued fractions with constant terms, it may be of advantage to
check directly the conditions of Theorem~\ref{thr:subk} instead of
bounding the term from below and from above.

\begin{corollary}
  \label{cor:constant-bis}
  Assume that $(\KK_h,\rho_h)$ is complete and condition
  \eqref{eq:xhc} holds. If $x\in\KK_h^*$ is such that, for some $k\geq
  1$,
  \begin{displaymath}
    x+[\underbrace{x,\dots,x}_{2k}]\geq ah
  \end{displaymath}
  and 
  \begin{displaymath}
    x+[\underbrace{x,\dots,x}_{2k-1}]\leq bh
  \end{displaymath}
  with $a>C_{b/a}$, then the continued fraction \eqref{eq:zn-xxx}
  converges. 
\end{corollary}

\begin{corollary}
  \label{cor:constant-bis-3}
  Assume that $(\KK_h,\rho_h)$ is complete and condition \eqref{eq:xh}
  holds. Let $z_n$ be the $n$th approximant of the continued fraction
  with constant term $x\in\KK_h^*$. If $z_{2k-1}\leq
  rh$ for $r<1$ and some $k\geq1$, then $z_n$ converges to
  $z\in\KK_h$.  
\end{corollary}

\begin{remark}
  The limit $z$ of the continued fraction with constant term $x$
  satisfies the equation
    \begin{equation}
    \label{eq:zlim}
    z^*=z+x\,,
  \end{equation}
  Consider now the changes that happen to the basic equation
  \eqref{eq:zlim} if either $z$ or $x$ are scaled. For each $t\geq 1$,
  $t^{-1}z$ satisfies
  \begin{displaymath}
    (t^{-1}z)^*=t^{-1}z+x_t
  \end{displaymath}
  for $x_t=(t-t^{-1})z^*+t^{-1}x$. Indeed,
  \begin{displaymath}
    t^{-1}z+x_t=t^{-1}(z+x)+(t-t^{-1})z^*=tz^*=(t^{-1}z)^*\,.
  \end{displaymath}
  Assume that \eqref{eq:xh} holds and $rh\leq x\leq Rh$ so that one of
  the conditions of Corollary~\ref{cor:uniform-simple} holds. Then
  the continued fraction with the constant term $tx$ converges
  for all $t\geq1$, so that there exists unique $z_t\in\KK$ that
  satisfies
  \begin{displaymath}
    z_t^*=z_t+tx,\qquad t\geq 1\,.
  \end{displaymath}
  Fix $\beta\in[0,1]$ and note that $y_t=t^\beta z_t$ satisfies the equation
  \begin{displaymath}
    y_t^*=t^{-2\beta}y_t+t^{1-\beta}x\,.
  \end{displaymath}
  Since $x\geq rh$, we have $y_t\leq t^{\beta-1}rh$. If $\beta<1$,
  then $y_t=t^\beta z_t$ converges to $\neutral$ as $t\to\infty$ by
  Lemma~\ref{lemma:order}.  If $\beta=1$, then the equation becomes
  $y_t^*=t^{-2}y_t+x$. Since 
  \begin{displaymath}
    t^{-2}y_t=t^{-1}z_t=t^{-1}(z_t+tx)^*\leq t^{-2}x^*\leq t^{-2}r^{-1}h\,,
  \end{displaymath}
  we have $(x+t^{-2}r^{-1}h)^*\leq tz_t\leq x^*$. Therefore, $tz_t$ converges
  to $x$ as $t\to\infty$.
\end{remark}

While the following result can be proved under condition
\eqref{eq:xhc}, we formulate its simpler version. 

\begin{theorem}
  \label{thr:dist-gen}
  Assume that \eqref{eq:xh} holds. If the continued fractions with
  constant terms $x'$ and $x''$, such that $x'\geq rh$ and $x''\geq rh$
  with $r>1$, converge respectively to $z'$ and $z''$, then
  \begin{displaymath}
    \rho_h(z',z'')\leq \rho_h(x',x'')\frac{1}{r^2-1}\,.
  \end{displaymath}
\end{theorem}
\begin{proof}
  It follows from \eqref{eq:zlim}, the triangle inequality, the
  translation invariance of the metric $\rho_h$, and
  Lemma~\ref{lemma:lip} that
  \begin{align*}
    \rho_h(z',z'')&=\rho_h((z'+x')^*,(z''+x'')^*)\\
    &\leq r^{-2}\rho_h(z'+x',z''+x'')\\
    &\leq r^{-2}\left(\rho_h(z',z'')+\rho_h(x',x'')\right)\,.
  \end{align*}
  The statement is obtained after rearranging the terms taking into
  account that $r>1$. 
\end{proof}

\begin{example}
  The setting of Theorem~\ref{thr:subk} is well adjusted to confirm
  the convergence of periodic continued fractions. For instance,
  consider the continued fraction with alternating elements
  $x,y\in\KK$ assuming that \eqref{eq:xh} holds. If $x\geq\eps h$,
  $y\geq \eps h$ for some $\eps>0$, and 
  \begin{displaymath}
    x+(y+x^*)^*\geq rh,\qquad y+(x+y^*)^*\geq rh
  \end{displaymath}
  for some $r>1$, then \eqref{eq:xmk} holds with $k=1$ and the
  continued fraction $[x,y,x,y,\ldots]$ converges.  If $y=x^*$, and
  \eqref{eq:xh} holds, then condition \eqref{eq:xmk} for sufficiently
  large $k$ amounts to $\upsilon(1,1)x\geq rh$ and
  $\upsilon(1,1)x^*\geq rh$ for some $r>1$, where
  $\upsilon(1,1)=\thf(1+\sqrt{5})$. In this case, the continued
  fraction $[x,x^*,x,x^*,\ldots]$ converges and the polar $y=z^*$ to
  its limit $z$ satisfies the equation
  \begin{displaymath}
    x+(x\oplus y)=y\,.
  \end{displaymath}
\end{example}

\section{Continued fractions of convex sets}
\label{sec:cont-fract-conv}

Let $\KK=\sCo$ be the family of all \emph{convex closed} sets in $\R^d$
containing the origin. We refer to \cite{schn2} for a wealth of
information about convex sets.  The closed \emph{Minkowski sum} $K+L$
of two sets $K,L\in\sCo$ is defined as the closure of the set
$\{x+y:\;x\in K,y\in L\}$ of pairwise sums of points from $K$ and
$L$. If at least one summand is compact, then the set of pairwise sums
is closed and no additional closure is required.

The family $\sCo$ with the closed Minkowski addition is a semigroup
with the neutral element $\neutral=\{0\}$ being the origin, partially
ordered by inclusion. Note that all inclusions for sets are understood
in the non-strict sense. Since the convex sets from $\sCo$ contain the origin,
$K\subset L$ implies that $K+M\subset L+M$ for $K,L,M\in\sCo$, so that
the order is compatible with the addition. The scaling by positive
reals is defined conventionally as $aK=\{ax:\; x\in K\}$.

It is known that the only order reversing involution on $\sCo$ (up to
a rigid motion) is the polar transform, see
\cite{art:mil08t,boer:sch08}. The \emph{polar} to $K\in\sCo$ is
defined as
\begin{displaymath}
  K^*=\{u\in\R^d:\; s_K(u)\leq 1\}\,,
\end{displaymath}
where 
\begin{displaymath}
  s_K(u)=\sup\{\langle u,x\rangle:\; x\in K\}
\end{displaymath}
is the \emph{support function} of $K$ and $\langle u,x\rangle$ denotes
the scalar product. Note that the support function may take infinite
values if $K$ is not bounded. Since the support function is
homogeneous of order 1, it suffices to consider its values for $u$
with the Euclidean norm $\|u\|=1$, i.e. for all $u$ from the unit
Euclidean sphere $\Sphere$ in $\R^d$. The inclusion of convex sets
turns into the pointwise domination of their support functions.  Let
\begin{displaymath}
  r_K(u)=\sup\{t:\; tu\in K\}
\end{displaymath}
be the radial function of $K\in\sCo$. Then $r_{K^*}(u)=1/s_K(u)$ for
all unit vectors $u$, see \cite[Sec.~1.6]{schn2}.

The only convex set invariant for the polar transform is the unit
Euclidean ball $B$.  If $K=rB$ is the ball of radius $r$ centred at
the origin, then $K^*=r^{-1}B$ is the ball of radius $r^{-1}$. Further
examples can be found in \cite[Sec.~1.6]{schn2}.  The polar to the
neutral element $\{0\}$ is the whole space, i.e. $\neutral^*=\R^d$.

The family $\KK_h=\sKo$ consists of \emph{convex bodies} (i.e. convex
compact sets) containing the origin and the metric
$\rho_h$ from \eqref{eq:rhoh} is the \emph{Hausdorff distance}
\begin{displaymath}
  \rho_H(K,L)=\inf\{\eps>0:\; K\subset L+\eps B,\; L\subset K+\eps B\}
\end{displaymath}
between $K$ and $L$ from $\sKo$. Note that $K+\eps B$ is called the
$\eps$-envelope of $K$, which is alternatively defined as the set of
points within distance at most $\eps$ to $K$.  The \emph{norm} of a
set defined as
\begin{displaymath}
  \|K\|=\sup\{\|x\|:\; x\in K\}=\rho_H(K,\{0\})
\end{displaymath}
is the radius of the smallest centred ball that contains $K$. The
family $\KK_h^*$ consists of all convex closed sets that contain a
neighbourhood of the origin and $\KK_h\cap \KK_h^*=\sK_{00}$ is the
family of convex bodies containing a neighbourhood of the origin.

It is known that $\sKo$ with the Hausdorff metric is a complete
separable metric space, see \cite[Th.~1.8.3]{schn2}.  We apply the
same definition of the Hausdorff metric also for closed (possibly
non-compact) sets, noticing that $\rho_H$ may take infinite values.

Consider the set-valued continued fraction given by
\begin{displaymath}
  F_n=[K_1,\dots,K_n], \quad n\geq 1\,,
\end{displaymath}
for a sequence $K_n\in\sCo$, $n\geq 1$. 

\begin{theorem}
  \label{thr:constant}
  Let $K_n=K$ for all $n\geq1$ and $K\in\sCo$. Assume that $K\supset rB$, where
  either (i) $r>1$ or (ii) $r=1$ and $K$ is compact or (iii) $0<r<1$
  and $\|K\|<r/(1-r)$. Then 
  \begin{equation}
    \label{eq:fnk}
    F_n=[\underbrace{K,\dots,K}_n]
  \end{equation}
  converges in the Hausdorff metric to a convex body $F\in\sKo$ that
  satisfies the equation
  \begin{equation}
    \label{eq:fplusk}
    F^*=F+K\,.
  \end{equation}
\end{theorem}
\begin{proof}
  We confirm the validity of \eqref{eq:xh} that becomes
  \begin{equation}
    \label{eq:xh-set}
    \rho_H(K^*,(K+tB)^*)\leq t
  \end{equation}
  for all $K\in\sCo$ such that $B\subset K$. Then $K^*\subset B$ and 
  \begin{align*}
    K^*&=\{u\in B:\; s_K(u)\leq 1\}\\
    &=\{u\in B:\; s_K(u)+s_{tB}(u)\leq 1+t\|u\|\}\\
    &\subset \{u\in B:\; s_K(u)+s_{tB}(u)\leq 1+t\}\\
    & =(1+t)\{u\in (1+t)^{-1}B:\; s_K(u)+s_{tB}(u)\leq 1\}\\
    & \subset (1+t) \{u\in B:\; s_{K+tB}(u)\leq 1\}\\
    & = (1+t)(K+tB)^*\,.
  \end{align*}
  For any convex body $A$ and any $\eps>0$, 
  \begin{displaymath}
    (1+\eps)A\subset A+\eps\|A\|B\,.
  \end{displaymath}
  Indeed, taking the support functions of the both sides, it is
  immediately seen that 
  \begin{displaymath}
    (1+\eps)s_A(u)\leq s_A(u)+\eps\|A\|\,,\quad u\in\Sphere\,.
  \end{displaymath}
  In view of this and the fact that $(K+tB)^*\subset B$,
  \begin{align*}
    K^*\subset (1+t)(K+tB)^*\subset (K+tB)^*+tB\,,
  \end{align*}
  so that \eqref{eq:xh-set} holds, and
  Corollary~\ref{cor:uniform-simple} yields the result.
\end{proof}

In view of Corollary~\ref{cor:inv-continuous}, condition
\eqref{eq:xh} verified in the proof of Theorem~\ref{thr:constant}
yields the following result that is of independent interest.

\begin{theorem}
  For any two convex compact sets $K,L$ containing the origin, 
  \begin{displaymath}
    \rho_H(K^*,L^*)\leq \max(\|K^*\|,\|L^*\|)^2\rho_H(K,L)\,.
  \end{displaymath}
\end{theorem}

\begin{theorem}
  \label{thr:nec-suf}
  Let $K$ be a convex set containing a neighbourhood of the
  origin. The continued fraction \eqref{eq:fnk} with constant term $K$
  converges in the Hausdorff metric if and only if $F_{2k-1}\subset
  aB$ for $a<1$ and at least one $k\geq1$.
\end{theorem}
\begin{proof}
  The sufficiency follows from
  Corollary~\ref{cor:constant-bis-3}. Assume that $F_n$ converges to
  $F$ that necessarily satisfies \eqref{eq:fplusk}.  Let $K\supset
  \eps B$ for $\eps>0$. Then $F\subset \eps^{-1}B$, whence $F$ is
  also compact and $\|F\|=R$ is finite. Then there exists $u$ with
  $\|u\|=1$ such that $s_F(u)=r_F(u)=R$. It follows from
  \eqref{eq:fplusk} that
  \begin{displaymath}
    s_{F^*}(u)\geq s_F(u)+\eps\,.
  \end{displaymath}
  Since $s_{F^*}(u)=1/r_F(u)=1/s_F(u)$, we have 
  \begin{displaymath}
    \frac{1}{s_F(u)}\geq s_F(u)+\eps\,,
  \end{displaymath}
  whence 
  \begin{displaymath}
    R=s_F(u)=\thf(\sqrt{\eps^2+4}-\eps)<1\,.
  \end{displaymath} 
  Thus, $F_n\subset aB$ for $a\in(R,1)$ and all sufficiently large
  $n$, in particularly, for the odd-numbered terms. 
\end{proof}

\begin{remark}
  Since the sequences $F_{2n}$ and $F_{2n-1}$ for the set-valued
  continued fraction with the constant term $K$ are monotone and
  bounded, they converge without any extra condition on the set
  $K$. But their limits may be different.
\end{remark}

\begin{example}
  a) It is easy to see that $K=rB$ satisfies the condition of
  Theorem~\ref{thr:constant} for each $r>0$. 

  b) A segment $K=[0,u]$ for a given $u\in\R^d$ does not satisfy the
  condition and the corresponding continued fraction diverges. Indeed,
  $F_n$ is a segment for even $n$ and a half-space for odd $n$.

  c) Assume that $K$ is a strip $\{(x_1,x_2):\; -a\leq x_1\leq
  a\}\subset\R^2$. Then $K^*$ is the segment with end-points at $(\pm
  a^{-1},0)$. Thus, $F_n$ is the segment with end-points at $(\pm
  a_n,0)$ where $a_n=[a,\dots,a]$. Thus, $F_n$ converges, whereas $K$
  does not satisfy the condition of Theorem~\ref{thr:constant} if
  $a\leq 1$. However, $F_{2k+1}\subset rB$ with $r<1$ for sufficiently
  large $k$, and so Theorem~\ref{thr:nec-suf} confirms the convergence
  of the continued fraction.
\end{example}

In relation to continued fractions generated by non-constant
sequences, Theorem~\ref{thr:subk} taking into account
Remark~\ref{rem:xmk} applies. Furthermore, the continued fraction
converges if $rB\subset K_n\subset RB$ for all $n\geq1$ and for $r$
and $R$ satisfying the conditions of
Corollary~\ref{cor:uniform-simple}. By
Theorem~\ref{thr:non-constant-R-r}, the continued fraction converges
if $r_nB\subset K_n$ and $\liminf_{n\to\infty} n\log(r_nr_{n+1})>1$.

\begin{example}
  \label{ex:seidel}
  Let $K$ and $L$ be two different centred segments in the plane. Then
  the infinite sum $K+L+K+L+\cdots$ is the whole plane. However, the
  continued fraction $[K,L,K,L,\dots]$ diverges. Indeed, $K^*$ is a
  strip, so that $L+K^*$ is the same strip of a different width, so
  that $(L+K^*)^*$ if a scaled variant of $K$, and the successive
  iterations result in a sequence that alternates between a scale of
  $K$ and a polar to it. This example shows that a direct
  generalisation of the Seidel--Stern theorem on continued fractions
  with positive terms \cite[Th.~4.28]{jon:thr80} fails in the
  set-valued case.
\end{example}

\begin{example}
  \label{ex:two-periodic}
  Convex sets that appear as limits of periodic continued fractions
  might be regarded as a generalisation of quadratic irrational
  numbers. Consider the continued fraction $[K,L,K,L,\ldots]$ with two
  alternating terms $K,L\in\sCo$. By Corollary~\ref{cor:epsilon-in}
  together with Remark~\ref{rem:xmk}, this continued fraction
  converges if $K+(L+K^*)^*\supset aB$ and $L+(K+L^*)^*\supset aB$ for
  $a>1$, and both $K$ and $L$ contain a neighbourhood of the
  origin. 
\end{example}

\begin{example}
  \label{ex:three-periodic}
  In the setting of Example~\ref{ex:two-periodic}, the continued
  fraction diverges if $K$ and $L$ are two segments, see also
  Example~\ref{ex:seidel}. In order to obtain a converging continued
  fraction built of segments, consider three centred non-collinear
  segments $L_i=[-u_i,u_i]$, $i=1,2,3$, in the plane and the
  corresponding continued fraction $[K_1,K_2,\ldots]$, where
  $K_{3n+i}=[-v_{3n+i},v_{3n+i}]=L_i$ for $n\geq 0$ and
  $i=1,2,3$. Condition \eqref{eq:xmk} for $k=1$ amounts to
  \begin{displaymath}
    K_n+(K_{n+1}+K_{n+2}^*)^*\supset aB
  \end{displaymath}
  with $a>1$.  A direct geometric calculation shows that
  \begin{displaymath}
    (K_{n+1}+K_{n+2}^*)^*=\frac{1}{1+|\langle v_{n+1},v_{n+2}\rangle|}
    [-v_{n+2},v_{n+2}]\,.
  \end{displaymath}
  Thus, \eqref{eq:xmk} holds if 
  \begin{equation}
    \label{eq:ui}
    [-u_{i_1},u_{i_1}]+\frac{1}{1+|\langle u_{i_2},u_{i_3}\rangle|}
    [-u_{i_3},u_{i_3}]\supset aB
  \end{equation}
  for some $a>1$ and all permutations $(i_1,i_2,i_3)$ of
  $(1,2,3)$. This is always possible to achieve by increasing the
  lengths of the segments. Furthermore, \eqref{eq:limsupm} holds with
  $l=1$ if
  \begin{equation}
    \label{eq:kper}
    [K_{n-2k},\dots,K_n]\subset RB,\qquad 
    [K_{n-2k},\dots,K_{n+1}]\subset RB,
  \end{equation}
  for a finite fixed $R$ and all sufficiently large $n$. In case of
  segments, this condition holds, since each of the approximants in
  \eqref{eq:kper} take only three possible values and all they are
  compact, since
  \begin{displaymath}
    [K_n,K_{n+1},K_{n+2}]^*=K_n+(K_{n+1}+K_{n+2}^*)^*
  \end{displaymath}
  contains a neighbourhood of the origin.
\end{example}

\begin{remark}
  The Minkowski sum in the definition of set-valued continued
  fractions can be replaced by other operations with sets, e.g. the
  convex hull of the union, the $L_p$-sum or the radial sum, see
  \cite{schn2}. For instance, if the convex hull of the union $K\vee
  M=\mathrm{conv}(K\cup M)$ is chosen as the semigroup operation, then
  $K\vee tB\subset K+tB$, so that \eqref{eq:xh} holds in this case and
  the convergence results from Section~\ref{sec:cont-fract-semigr}
  apply.  Furthermore, it is possible to consider a sequence of
  alternating operations, e.g. the Minkowski sum and the radial
  sum. The latter case is particularly easy, since its reduces to the
  numerical continued fraction built of the values of the support
  function of the terms, so that the classical convergence criteria
  apply.
\end{remark}

\section{Space of non-negative convex functions}
\label{sec:space-non-negative}
\label{sec:legendre-transform}

Let $\KK=\Cvx_0(\R^d)$ be the space of \emph{convex functions}
$f:\R^d\mapsto[0,\infty]$ such that $f(0)=0$ with the arithmetic
addition as the semigroup operation and the pointwise partial order.
The \emph{Legendre--Fenchel transform} of a function $f$ is defined as
\begin{displaymath}
  f^*(x)=\sup_y (\langle x,y\rangle - f(y)),\quad x\in\R^d\,.
\end{displaymath}
It is well known that the Legendre--Fenchel transform is an order
reversing involution on $\Cvx_0(\R^d)$, see \cite{roc70}. The neutral
element is the function identically equal to zero, and its dual is the
convex function identically equal to infinity outside the origin. The
only self-polar function is $h(x)=\thf \|x\|^2$. The family $\KK_h$ is
the family of convex functions that admit a quadratic majorant, and the
metric $\rho_h$ is given by
\begin{displaymath}
  \rho_h(f,g)=\inf\{\eps>0:\; f(x)\leq g(x)+\frac{\eps}{2}\|x\|^2,\;
  g(x)\leq f(x)+\frac{\eps}{2}\|x\|^2,\; x\in\R^d\}\,.
\end{displaymath}
It is easy to see that \eqref{eq:3} holds and so $\rho_h$ is indeed a
metric. 

The dual operation (see Remark~\ref{rem:dual-addition}) to the
arithmetic addition is the \emph{inf-convolution}
\begin{displaymath}
  (f\oplus g)(x)=\inf_{x_1+x_2=x} (f(x_1)+g(x_2))\,.
\end{displaymath}

Consider the continued fraction 
\begin{displaymath}
  z_n=[f_1,\dots,f_n]\in\Cvx_0(\R^d)\,,\quad n\geq 1\,,
\end{displaymath}
generated by a sequence $\{f_n,n\geq1\}$ from $\Cvx_0(\R^d)$. 

\begin{theorem}
  \label{thr:legendre}
  Assume that a function $f\in\Cvx_0(\R^d)$ satisfies
  $\frac{r}{2}\|x\|^2\leq f(x)\leq \frac{R}{2}\|x\|^2$ for all
  $x\in\R^d$, where $R\in[r,\infty]$, and
  \begin{equation}
    \label{eq:func-rr}
    r^2+\frac{4r}{R}>4\,.
  \end{equation}
  Then the continued fraction $z_n=[f,\dots,f]$ with the constant term
  $f$ converges in the metric $\rho_h$ to $z\in \KK_h$ satisfying
  $z^*=z+f$. 
\end{theorem}
\begin{proof}
  It is easy to see that \eqref{eq:xhc} holds with $C_R\leq
  C_\infty=2$.  Indeed, if $f\geq h$, then
  \begin{displaymath}
    \langle x,x+v\rangle - f(x+v)
    \leq \langle x,x+v\rangle -\thf \|x+v\|^2
    =\thf \|x\|^2-\thf \|v\|^2\,,
  \end{displaymath}
  whence
  \begin{align*}
    f^*(x)&=\sup_{v\in\R^d} [\langle x,x+v\rangle - f(x+v)]\\
    &=\sup_{\|v\|\leq \|x\|} [\langle x,x+v\rangle - f(x+v)]\\
    &=\sup_{\|v\|\leq \|x\|} [\langle x,x+v\rangle - f(x+v)
    -\thf t\|x+v\|^2+\thf t\|x+v\|^2]\\
    &\leq \sup_{v\in\R^d} [\langle x,x+v\rangle - f(x+v)
    -\thf t\|x+v\|^2+2 t\|x\|^2]\\
    &= (f+th)^*(x)+4 th\,.
  \end{align*}
  Thus, $\rho_h(f^*,(f+th)^*)\leq 4t$. In order to improve the
  inequality, start by observing that
  \begin{align*}
    \rho_h(f^*,(f+th)^*)
    &=\inf\{\eps>0:\; f^*\leq (f+th)^*+\eps h\}\\
    &=\inf\{\eps>0:\; f^*\leq (f^*\oplus t^{-1}h)+\eps h\}\,.
  \end{align*}
  If $h\leq f\leq Rh$, then $R^{-1}h\leq f^*\leq h$. Therefore,
  \begin{displaymath}
    f^*(x-y)\geq f^*(x)-k_x\|y\|\,,
  \end{displaymath}
  where $k_x=\|x\|(1+\sqrt{1-R^{-1}})$. Indeed, $k_x$ is the steepest
  slope of the tangent line to the graph of $\thf \|u\|^2$,
  $u\in\R^d$, that passes through the point $(x,\thf
  R^{-1}\|x\|^2)\in\R^{d+1}$. Thus,
  \begin{align*}
    (f^*\oplus t^{-1}h)(x)
    &=\inf 
    (f^*(x-y)+\thf t^{-1}\|y\|^2)\\
    &\geq f^*(x) + \inf
    (-k_x\|y\|+\thf t^{-1}\|y\|^2)\\
    &\geq f^*(x)-\thf t \|x\|^2(1+\sqrt{1-R^{-1}})^2\,.
  \end{align*}
  Thus, \eqref{eq:xhc} holds with 
  \begin{equation}
    \label{eq:crf}
    C_R=1+\sqrt{1-R^{-1}}\,.
  \end{equation}
  Finally, \eqref{eq:urr} holds if
  \begin{displaymath}
    1+\sqrt{1-\frac{\upsilon(r,R)}{\upsilon(R,r)}}
    <\upsilon(r,R)\,,
  \end{displaymath}
  which is equivalent to the imposed condition \eqref{eq:func-rr}.

  It remains to show that $\KK_h$ is complete in the metric
  $\rho_h$. If $\{f_n\}$ is a fundamental sequence in $\rho_h$, then
  $f_n(x)$ is a fundamental sequence for each $x\in\R^d$, so that
  $f_n(x)\to f(x)$ for all $x$ and $f\in\KK_h$. Finally, $f_n\leq
  f_m+\eps h$ for any $\eps>0$ and all sufficiently large $n$ and $m$
  implies that $f\leq f_m+\eps h$ and $f_n\leq f+\eps h$ by letting
  $n\to\infty$ and $m\to\infty$.
\end{proof}

Given that \eqref{eq:xhc} holds with $C_R$ given by \eqref{eq:crf},
all results from Section~\ref{sec:convergence-results} can be used to
obtain further sufficient conditions for the convergence of continued
fractions in $\Cvx_0(\R^d)$. For instance,
Corollary~\ref{cor:constant-bis} implies that the continued fraction
with constant term $f$ converges if its approximants satisfy
$z_{2k+1}\leq a^{-1}h$ and $z_{2k}\geq b^{-1}h$ for some $k\geq1$ with
$a>C_{b/a}$. The latter condition amounts to $a>2-b^{-1}$. 

\bigskip

The other polarity transform (A-transform) on $\Cvx_0(\R^d)$
thoroughly analysed in \cite{art:mil11h} is given by
\begin{equation}
  \label{eq:atran}
  f^o(x)=
  \begin{cases}
    \sup \left\{\frac{\langle x,y\rangle-1}{f(y)}:\; y\in\R^d,\; f(y)>0\right\}, &
    x\in \{f^{-1}(0)\}^*\,,\\
    \infty, & \text{otherwise}\,.
  \end{cases}
\end{equation}
For simplicity, assume that $d=1$.  In this case the family of
self-polar functions includes
\begin{displaymath}
  h_p(x)=\left(\frac{(p-1)^{p-1}}{p^p}\right)^{\thf} \|x\|^p
\end{displaymath}
for any $p\in[1,\infty]$, see \cite{rotem12}.

For any finite positive outside the origin self-polar function $h$,
condition \eqref{eq:xh} holds. Indeed, if $f\geq h$, then
\begin{align*}
  f^o(x)&=\sup_{y\neq 0} \left[\frac{\langle x,y\rangle-1}{f(y)}
    -\frac{\langle x,y\rangle-1}{f(y)+th(y)}
    +\frac{\langle x,y\rangle-1}{f(y)+th(y)}\right]\\
  &\leq \sup_{y\neq 0} \frac{th(y)(\langle x,y\rangle-1)}{f(y)(f(y)+th(y))}
  +(f+th)^*(x)\\
  &\leq t\sup_{y\neq 0} \frac{\langle x,y\rangle-1}{h(y)}+(f+th)^*(x)\\
  &=th(x)+(f+th)^*(x)\,.
\end{align*}
Since $\KK_h$ is complete with the $\rho_h$ metric, 
Corollary~\ref{cor:uniform-simple} yields the convergence of continued
fractions with constant terms and further results from
Section~\ref{sec:convergence-results} 
apply in this case.

\section{Acknowledgements}
\label{sec:acknowledgements}

The author is grateful to the Department of Statistics of the
University Carlos III of Madrid for hospitality in 2012 when this work
commenced and to Bernardo d'Auria for numerous discussions at early
stages. The author has also benefited from discussions with Sergei
Foss, Daniel Hug, Takis Konstantopoulos and Matthias Reitzner.

\end{document}